\documentclass[leqno,12pt]{article} 
\usepackage{amsmath, amssymb}
\usepackage{amsthm} 
\usepackage{pictexwd,dcpic}
\usepackage{multirow}
\theoremstyle{plain} 
\newtheorem{theorem}{\indent\sc Theorem}[section]
\newtheorem{lemma}[theorem]{\indent\sc Lemma}

\newtheorem{proposition}[theorem]{\indent\sc Proposition}

\theoremstyle{definition} 

\newtheorem{remark}[theorem]{\indent\sc Remark}
\newtheorem{example}[theorem]{\indent\sc Example}

\makeatletter
\def\address#1#2{\begingroup
\noindent\parbox[t]{7.8cm}{%
\small{\scshape\ignorespaces#1}\par\vskip1ex
\noindent\small{\itshape E-mail address}%
\/: #2\par\vskip4ex}\hfill%
\endgroup}%
\makeatother
\title{On the Diophantine equation $pq=x^2+ny^2$} 
\author{
\textsc{Ja Kyung Koo and Dong Hwa Shin$^*$} 
}
\date{} 
\begin{document}

\maketitle

\footnote{ 
2010 \textit{Mathematics Subject Classification}. Primary 11E25; Secondary 11E16.}
\footnote{ 
\textit{Key words and phrases}. Binary quadratic forms, representations.} \footnote{
\thanks{
The first named author was partially supported by the NRF of Korea
grant funded by MISP (2013042157). $^*$The corresponding author was
supported by Hankuk University of Foreign Studies Research Fund of
2014.} }

\begin{abstract}
Let $n$ be a positive integer.
We investigate pairs of distinct odd primes $p$ and $q$
not dividing $n$ for which
the Diophantine equations $pq=x^2+ny^2$ have integer solutions in $x$ and $y$.
As its examples we classify all such pairs of $p$ and $q$ when $n=5$ and $14$.
 \end{abstract}

\maketitle

\section {Introduction}\label{intro}

In a letter to Blaise Pascal dated September 25 in 1654
Fermat announced that for primes $p$ and $q$
\begin{equation}\label{Fermat}
p,q\equiv3,7\pmod{20}~\Longrightarrow~pq=x^2+5y^2~\textrm{for some}~x,y\in\mathbb{Z}.
\end{equation}
And, Lagrange \cite{Lagrange} first noted that
\begin{equation*}
p\equiv3,7\pmod{20}~\Longrightarrow~p=2x^2+2xy+3y^2~\textrm{for some}~x,y\in\mathbb{Z}.
\end{equation*}
He then presented the identity
\begin{equation}\label{Lag}
(2x^2+2xy+3y^2)(2z^2+2zw+3w^2)
=(2xz+xw+yz+3yw)^2+5(xw-yz)^2,
\end{equation}
which completes the proof. 
\par
On the other hand, what can we say about the other direction of (\ref{Fermat})?
As far as we understand, not much has been known in this theme so far.
More generally, let $n$ be a positive integer, and let $p$ and $q$ be distinct odd primes not dividing $n$. In this paper we shall find an equivalent condition under which the product $pq$ can be represented by the principal form $x^2+ny^2$ (Theorem \ref{main}). To this end,
we shall review some classical facts about
form class groups and ideal class groups ($\S$2--3),
from which we are able to give a satisfactory answer to the question above and
to the solvability of the Diophantine equation $pq=x^2+14y^2$ (Examples \ref{example1} and \ref{example2}) as well.

\section {Binary quadratic forms}\label{section2}

We first introduce a general version of the identity (\ref{Lag}).

\begin{lemma}\label{extendLag}
Let $a,b,c\in\mathbb{Z}$ with $n=-b^2+ac$.
We have the identity
\begin{equation*}
(ax^2+2bxy+cy^2)(az^2+2bzw+cw^2)=(axz+bxw+byz+cyw)^2
+n(xw-yz)^2.
\end{equation*}
\end{lemma}
\begin{proof}
Immediate.
\end{proof}

Let $f(x,y)=ax^2+bxy+cy^2$ be an integral binary quadratic form (for short, a form).
It is said to be \textit{primitive} if its coefficients $a$, $b$ and $c$ are relatively prime.
Let $m$ be an integer. We say that $m$ is \textit{represented} by $f(x,y)$ if the equation
$m=f(x,y)$ has an integer solution in $x$ and $y$. Moreover, if such $x$ and $y$ are relatively prime, then we say that $m$ is \textit{properly represented} by $f(x,y)$.
\par
Let $f(x,y)$ and $g(x,y)$ be primitive positive definite forms of
the same discriminant. We say that they are \textit{equivalent} (respectively, \textit{properly equivalent})
if there is an element $\left[\begin{matrix}p&q\\r&s\end{matrix}\right]$ of $\mathrm{GL}_2(\mathbb{Z})$ (respectively, $\mathrm{SL}_2(\mathbb{Z})$) such that
\begin{equation*}
f(x,y)=g(px+qy,rx+sy).
\end{equation*}
The equivalence and properly equivalence of forms are indeed equivalence relations. Furthermore, two equivalent forms represent the same numbers \cite[pp.24--25]{Cox}.
\par
For a negative integer $D$ such that $D\equiv0,1\pmod{4}$ we let $C(D)$ be the set of all
properly equivalence classes of primitive positive definite forms of discriminant $D$.
As is well-known the Dirichlet composition makes $C(D)$ into a finite abelian group,
called the \textit{form class group} of discriminant $D$ \cite[Theorem 3.9]{Cox}.
We denote by $h(D)$ the order of the group $C(D)$.
\par
We call a primitive positive definite form $ax^2+bxy+cy^2$
a \textit{reduced form} if the coefficients $a$, $b$ and $c$ satisfy
\begin{eqnarray*}
|b|\leq a\leq c,~\textrm{and}~b\geq0~\textrm{if either}~|b|=0~\textrm{or}~a=c.
\end{eqnarray*}
Then, every primitive positive definite form of discriminant $D$ is properly equivalent to
a unique reduced form of the same discriminant \cite[Theorem 2.8]{Cox}.
Hence $h(D)$ is equal to the number of reduced forms of discriminant $D$.

\begin{proposition}\label{Legendre}
Let $n$ be a positive integer and $p$ be an odd prime not dividing $n$.
Then we get the assertion
\begin{equation*}
\bigg(\frac{-n}{p}\bigg)=1~\Longleftrightarrow~
\textrm{$p$ is represented by a reduced form of discriminant $-4n$}.
\end{equation*}
\end{proposition}
\begin{proof}
See \cite[Corollary 2.6]{Cox}.
\end{proof}

Let $n$ be a positive integer. Euler called $n$ a \textit{convenient number} if it satisfies the following properties:
\begin{quote}
Let $m$ be an odd positive integer relatively prime to $n$
which is properly represented by $x^2+ny^2$.
If the equation $m=x^2+ny^2$ has only one solution with $x,y\geq0$, then $m$ is a prime.
\end{quote}
We say that two primitive positive definite forms of discriminant $D$ ($<0$)
are in the same \textit{genus} if they represent the same values in $(\mathbb{Z}/D\mathbb{Z})^\times$.
Then it is well-known that $n$ is a convenient number if and only if every genus
of discriminant $-4n$ consists of a single properly equivalence class
\cite[Proposition 3.24]{Cox}
(see also \cite{Kani}, \cite{Weil} and \cite{Weinberger}).

\section {Ideal class groups}\label{section3}

Let $K$ be an imaginary quadratic field of discriminant $d_K$
and $\mathcal{O}_K$ be its ring of integers.
Let $\mathcal{O}$ be the order conductor $f$ ($>0$) in $K$, so
its discriminant $D$ satisfies the relation $D=f^2d_K$.
We denote by
$I(\mathcal{O})$ the group of all proper fractional $\mathcal{O}$-ideals (under multiplication),
and by $P(\mathcal{O})$ its subgroup of all principal $\mathcal{O}$-ideals.
We call
the associated quotient group $C(\mathcal{O})=I(\mathcal{O})/P(\mathcal{O})$
the \textit{ideal class group} of the order $\mathcal{O}$.
Then we will relate this ideal class group $C(\mathcal{O})$ with the form class group
$C(D)$ defined in $\S$\ref{section2} via the following isomorphism
\begin{equation}\label{C(D)C(O)}
\begin{array}{ccc}C(D)&\stackrel{\sim}{\rightarrow}&C(\mathcal{O})\\
\textrm{class of}~ax^2+bxy+cy^2&\mapsto&
\textrm{class of}~[a,(-b+\sqrt{D})/2]
\end{array}
\end{equation}
\cite[Theorem 7.7(ii)]{Cox}.

\begin{lemma}\label{norm}
Let $f(x,y)$ be a reduced form of discriminant $D$ and $m$ be a positive integer. Then,
\begin{eqnarray*}
\textrm{$m$ is represented by $f(x,y)$}
&\Longleftrightarrow&m=N_\mathcal{O}(\mathfrak{a})~(=|\mathcal{O}/\mathfrak{a}|)~\textrm{for some
$\mathcal{O}$-ideal $\mathfrak{a}$}\\
&&\textrm{in the corresponding ideal class in $C(\mathcal{O})$}.
\end{eqnarray*}
\end{lemma}
\begin{proof}
See \cite[Theorem 7.7(iii)]{Cox}.
\end{proof}

Furthermore, let $I(\mathcal{O},f)$ be the subgroup of
$I(\mathcal{O})$ consisting of all proper factional $\mathcal{O}$-ideals prime to $f$, and
let $I_K(f)$ be the group of all fractional ideals of $K$ prime to $f$.
Then we have an isomorphism
\begin{equation}\label{O_KO}
\begin{array}{ccc}
\phi_f~:~I_K(f)&\stackrel{\sim}{\rightarrow}&I(\mathcal{O},f)\\
\mathfrak{a}&\mapsto&\mathfrak{a}\cap\mathcal{O}
\end{array}
\end{equation}
which preserves norms, namely $N_{\mathcal{O}_K}(\mathfrak{a})=N_\mathcal{O}(\mathfrak{a}\cap\mathcal{O})$.
Its inverse map $\phi_f^{-1}$ is given by
\begin{equation*}
\phi_f^{-1}(\mathfrak{b})=\mathfrak{b}\mathcal{O}_K\quad(\mathfrak{b}\in I(\mathcal{O},f))
\end{equation*}
\cite[Proposition 7.20]{Cox}.
\par
On the other hand, we know by the existence theorem
of class field theory that
there
exists a unique abelian extension $H_\mathcal{O}$ of $K$ whose
Galois group satisfies
\begin{equation*}
\mathrm{Gal}(H_\mathcal{O}/K)\simeq C(\mathcal{O})
\end{equation*}
(\cite[Chapters IV and V]{Janusz} and
\cite[Proposition 7.22]{Cox}).
We call this extension $H_\mathcal{O}$ the \textit{ring class field} of the order $\mathcal{O}$.

\begin{proposition}\label{principalform}
Let $n$ be a positive integer
and $H_\mathcal{O}$
be the ring class field of the order $\mathcal{O}=[1,\sqrt{-n}]$
in the imaginary quadratic field $K=\mathbb{Q}(\sqrt{-n})$.
Let $f_n(x)$ be the minimal polynomial of a real algebraic integer
which generates $H_\mathcal{O}$ over $K$.
If an odd prime $p$ divides neither $n$ nor the discriminant of $f_n(x)$, then
\begin{eqnarray*}
p=x^2+ny^2~\textrm{for some}~x,y\in\mathbb{Z}&\Longleftrightarrow&
\bigg(\frac{-n}{p}\bigg)=1~\textrm{and}~f_n(x)\equiv0\pmod{p}\\
&&\textrm{has an integer solution.}
\end{eqnarray*}
\end{proposition}
\begin{proof}
See \cite[Theorem 9.2]{Cox}.
\end{proof}

\begin{remark}
By the main theorem of complex multiplication, the singular value $j(\mathcal{O})$
of the elliptic modular function $j$ generates $H_\mathcal{O}$ over $K$ as
a real algebraic integer \cite[Chapter 5, Theorem 4 and Chapter 10, Theorem 5]{Lang}.
However, its minimal polynomial has too large integer coefficients for practical use (such as Proposition \ref{principalform}).
On the other hand, one can refer to \cite{Gee}, \cite{K-S-Y} and \cite{Schertz}
for other smaller generators of $H_\mathcal{O}$
in terms of the singular values of
Weber functions and eta-quotients.
\end{remark}

\section {Products of two primes in the form $x^2+ny^2$}

We are ready to prove our main theorem concerning the
Diophantine equation $pq=x^2+ny^2$.

\begin{theorem}\label{main}
Let $n$ be a positive integer, and let $p$ and $q$ be
distinct odd primes not dividing $n$. Then,
\begin{eqnarray*}
pq=x^2+ny^2~\textrm{for some}~x,y\in\mathbb{Z}&\Longleftrightarrow&
\textrm{there is a reduced form of discriminant $-4n$}\\
 &&\textrm{representing both $p$ and $q$}.
\end{eqnarray*}
\end{theorem}
\begin{proof}
First, assume that there exists a reduced form
$f(x,y)
=ax^2+bx+c$ of discriminant $-4n$ representing
both $p$ and $q$. Then we have
\begin{equation*}
p=f(x_1,y_1)~\textrm{and}~q=f(x_2,y_2)~\textrm{for some}~x_1,y_1,x_2,y_2\in\mathbb{Z}.
\end{equation*}
It follows from $b^2-4ac=-4n$ that $b$ is an even integer. So we achieve
by Lemma \ref{extendLag}
\begin{equation*}
pq=f(x_1,y_1)f(x_2,y_2)=
(ax_1x_2+(b/2)x_1y_2+(b/2)y_1x_2+cy_1y_2)^2
+n(x_1y_2-y_1x_2)^2.
\end{equation*}
This yields that $pq$ can be written in the form $x^2+ny^2$ for some $x,y\in\mathbb{Z}$.
\par
Conversely, assume that
\begin{equation}\label{assumption}
pq=x^2+ny^2~\textrm{for some}~x,y\in\mathbb{Z}.
\end{equation}
Since $p$ and $q$ are distinct odd primes not dividing $n$ by hypothesis, we deduce from (\ref{assumption})
\begin{equation*}
\bigg(\frac{-n}{p}\bigg)=
\bigg(\frac{-n}{q}\bigg)=1.
\end{equation*}
This implies that both $p$ and $q$ split in the imaginary quadratic field
$K=\mathbb{Q}(\sqrt{-n})$. Let
\begin{equation}\label{pq}
p\mathcal{O}_K=\mathfrak{p}\overline{\mathfrak{p}}~\textrm{and}~
q\mathcal{O}_K=\mathfrak{q}\overline{\mathfrak{q}}
\end{equation}
be the prime ideal factorizations of $p\mathcal{O}_K$ and $q\mathcal{O}_K$, respectively. We then
derive that
\begin{eqnarray*}
(pq)\mathcal{O}_K
&=&(x^2+ny^2)\mathcal{O}_K\quad\textrm{by (\ref{assumption})}\\
&=&(x+\sqrt{-n}y)\mathcal{O}_K\overline{(x+\sqrt{-n}y)\mathcal{O}_K}\\
&=&(p\mathcal{O}_K)(q\mathcal{O}_K)\\
&=&(\mathfrak{p}\overline{\mathfrak{p}})(\mathfrak{q}\overline{\mathfrak{q}})\quad\textrm{by (\ref{pq})}\\
&=&(\mathfrak{p}\mathfrak{q})\overline{(\mathfrak{p}\mathfrak{q})}\quad(=
(\mathfrak{p}\overline{\mathfrak{q}})
\overline{(\mathfrak{p}\overline{\mathfrak{q}})}).
\end{eqnarray*}
Thus without loss of generality we may assume
by the uniqueness of prime ideal factorization in $\mathcal{O}_K$
that
\begin{equation}\label{pqprincipal}
\mathfrak{p}\mathfrak{q}=(x+\sqrt{-n}y)\mathcal{O}_K.
\end{equation}
\par
On the other hand, let $\mathcal{O}=[1,\sqrt{-n}]$, which is the order in $K$
of discriminant $D=f^2d_K=-4n$.
Here, $d_K$ is the discriminant of $K$ and $f$ is the conductor of $\mathcal{O}$.
Let $f(x,y)$ be the reduced form of discriminant $D$ corresponding to the ideal class
$C=[\overline{\mathfrak{p}}\cap\mathcal{O}]$ in $\mathrm{C}(\mathcal{O})$
via the isomorphism in (\ref{C(D)C(O)}).
Then the norm
$N_\mathcal{O}(\overline{\mathfrak{p}}\cap\mathcal{O})=N_{\mathcal{O}_K}(\overline{\mathfrak{p}})=p$
can be represented by $f(x,y)$ by Lemma \ref{norm}. Furthermore,
we attain that in $C(\mathcal{O})=I(\mathcal{O})/P(\mathcal{O})$,
\begin{eqnarray*}
[\mathfrak{q}\cap\mathcal{O}]&=&[p\mathcal{O}][\mathfrak{q}\cap\mathcal{O}]
\quad\textrm{since}~p\mathcal{O}\in P(\mathcal{O})\\
&=&[p\mathcal{O}_K\cap\mathcal{O}][\mathfrak{q}\cap\mathcal{O}]\\
&&\textrm{by considering the isomorphism $\phi_f$ in (\ref{O_KO}) and its inverse map $\phi_f^{-1}$}\\
&=&[(p\mathcal{O}_K)\mathfrak{q}\cap\mathcal{O}]\quad\textrm{by the homomorphism property of $\phi_f$}\\
&=&[\mathfrak{p}\overline{\mathfrak{p}}\mathfrak{q}\cap\mathcal{O}]\quad\textrm{by (\ref{pq})}\\
&=&[\mathfrak{p}\mathfrak{q}\cap\mathcal{O}][\overline{\mathfrak{p}}\cap\mathcal{O}]
\quad\textrm{again by the homomorphism property of $\phi_f$}\\
&=&[(x+\sqrt{-n}y)\mathcal{O}_K\cap\mathcal{O}][\overline{\mathfrak{p}}\cap\mathcal{O}]\quad
\textrm{by (\ref{pqprincipal})}\\
&=&[(x+\sqrt{-n}y)\mathcal{O}][\overline{\mathfrak{p}}\cap\mathcal{O}]\quad
\textrm{because}~x+\sqrt{-n}y\in\mathcal{O}\\
&=&[\overline{\mathfrak{p}}\cap\mathcal{O}]\quad\textrm{since}~(x+\sqrt{-n}y)\mathcal{O}\in P(\mathcal{O})\\
&=&C.
\end{eqnarray*}
Therefore the norm $N_\mathcal{O}(\mathfrak{q}\cap\mathcal{O})
=N_{\mathcal{O}_K}(\mathfrak{q})=q$ is also represented by $f(x,y)$ by Lemma \ref{norm}. This completes the proof of the theorem.
\end{proof}

\section {Examples}

In this section we shall present a couple of concrete examples of Theorem \ref{main}.

\begin{example}\label{example1}
First we consider the question raised in $\S$\ref{intro} about
the Diophantine equation $pq=x^2+5y^2$.
Let $n=5$. There are
two reduced forms of discriminant $-4n=-20$, namely
\begin{equation*}
x^2+5y^2~\textrm{and}~2x^2+2xy+3y^2.
\end{equation*}
One can then readily check that
\begin{equation*}
\left\{\begin{array}{cccl}
x^2+5y^2&\textrm{represents}&1,9&\textrm{in}~(\mathbb{Z}/20\mathbb{Z})^\times,\\
2x^2+2xy+3y^2&\textrm{represents}&3,7&\textrm{in}~(\mathbb{Z}/20\mathbb{Z})^\times
\end{array}\right.
\end{equation*}
\cite[(2.20)]{Cox}.
This shows that there are two genera of discriminant $-20$
each consisting of a single class,
and hence $n=5$ is a convenient number.
\par
On the other hand, we obtain by Proposition \ref{Legendre} that for an odd prime $p$ other than $5$
\begin{eqnarray*}
p=x^2+5y^2~\textrm{or}~2x^2+2xy+3y^2~\textrm{for some}~x,y\in\mathbb{Z}
&\Longleftrightarrow&\bigg(\frac{-5}{p}\bigg)=1\\
&\Longleftrightarrow&p\equiv1,3,7,9\pmod{20}.
\end{eqnarray*}
Therefore we conclude by Theorem \ref{main} that for distinct odd primes $p$ and $q$ other than $5$
\begin{eqnarray*}
pq=x^2+5y^2~\textrm{for some}~x,y\in\mathbb{Z}
&\Longleftrightarrow&
\textrm{both $p$ and $q$ are represented by}~
\left\{\begin{array}{l} x^2+5y^2,~\textrm{or}\\
2x^2+2xy+3y^2\end{array}\right.
\\
&\Longleftrightarrow&
\left\{\begin{array}{l}
p,q\equiv1,9\pmod{20},~\textrm{or}\\
p,q\equiv3,7\pmod{20}.
\end{array}\right.
\end{eqnarray*}
\end{example}

\begin{example}\label{example2}
Let $n=14$. We then have four reduced forms of discriminant $-4n=-56$
\begin{equation*}
x^2+14y^2,~2x^2+7y^2~\textrm{and}~3x^2\pm2xy+5y^2,
\end{equation*}
and we see that
\begin{equation}\label{twogenera}
\left\{\begin{array}{cccl}
x^2+14y^2,~2x^2+7y^2&\textrm{represents}&1,9,15,23,25,39&\textrm{in}~(\mathbb{Z}/56\mathbb{Z})^\times,\\
3x^2\pm2xy+5y^2&\textrm{represents}&3,5,13,19,27,45&\textrm{in}~(\mathbb{Z}/56\mathbb{Z})^\times
\end{array}\right.
\end{equation}
\cite[(2.21)]{Cox}. This shows that $n=14$ is a non-convenient number.
Observe that $3x^2+2xy+5y^2$ and $3x^2-2xy+5y^2$ are equivalent, but $x^2+5y^2$ and $2x^2+7y^2$ are not.
Furthermore, we obtain by Proposition \ref{Legendre} that for an odd prime $p$ not dividing $14$
\begin{eqnarray*}
&&\textrm{$p$ is represented by one of the above reduced forms}\\
&\Longleftrightarrow&
\bigg(\frac{-14}{p}\bigg)=1\\
&\Longleftrightarrow&p\equiv
1,3,5,9,13,15,19,23,25,27,39,45\pmod{56}.
\end{eqnarray*}
\par
Let $K=\mathbb{Q}(\sqrt{-14})$ and
$\mathcal{O}=\mathcal{O}_K=[1,\sqrt{-14}]$.
Then we know that $H_\mathcal{O}$ is generated by a real algebraic integer
$\alpha=\sqrt{2\sqrt{2}-1}$ whose
minimal polynomial over $K$ is $(x^2+1)^2-8$ with discriminant $-2^{14}\cdot7$ \cite[Proposition 5.31]{Cox}.
Thus we deduce by Proposition \ref{principalform} that
for an odd prime $p$ not dividing $14$
\begin{equation}\label{x2+14y2}
p=x^2+14y^2~\textrm{for some}~x,y\in\mathbb{Z}~
\Longleftrightarrow~\bigg(\frac{-14}{p}\bigg)=1~\textrm{and}~p\in S,
\end{equation}
where
\begin{equation*}
S=\{\textrm{primes}~p~|~(x^2+1)^2\equiv8\pmod{p}~\textrm{has an integer solution}\}
\end{equation*}
\cite[Theorem 5.33]{Cox}.
\par
Suppose that
there exists an odd prime $p$ not dividing $14$
which is represented by $x^2+14y^2$ and $2x^2+7y^2$, simultaneously.
Let
\begin{eqnarray}
&&p=x_1^2+14y_1^2\quad\textrm{for some}~x_1,y_1\in\mathbb{Z},\label{first}\\
&&p=2x_2^2+7y_2^2\quad\textrm{for some}~x_2,y_2\in\mathbb{Z}.\label{second}
\end{eqnarray}
Here, $x_1,y_1,x_2,y_2$ are all nonzero integers modulo $p$
because $p$ is relatively prime to $14$.
We claim by (\ref{first}) that
$p$ splits in $K=\mathbb{Q}(\sqrt{-14})$ and
\begin{equation}\label{pfactorization}
p\mathcal{O}_K=\mathfrak{p}\overline{\mathfrak{p}}\quad\textrm{with}~
\mathfrak{p}=(x_1+\sqrt{-14}y_1)\mathcal{O}_K~\textrm{and}~
\overline{\mathfrak{p}}=(x_1-\sqrt{-14}y_1)\mathcal{O}_K.
\end{equation}
Moreover, since $2$ is ramified in $K$, we have
\begin{equation}\label{2factorization}
2\mathcal{O}_K=\mathfrak{a}\overline{\mathfrak{a}}~\textrm{with}~\mathfrak{a}=\overline{\mathfrak{a}}.
\end{equation}
We then see that
\begin{eqnarray*}
(2p)\mathcal{O}_K&=&
(4x_2^2+14y_2^2)\mathcal{O}_K\quad\textrm{by (\ref{second})}\\
&=&(2x_2+\sqrt{-14}y_2)\mathcal{O}_K\overline{(2x_2+\sqrt{-14}y_2)\mathcal{O}_K}\\
&=&(2\mathcal{O}_K)(p\mathcal{O}_K)\\
&=&(\mathfrak{a}\overline{\mathfrak{a}})(\mathfrak{p}\overline{\mathfrak{p}})\quad
\textrm{by (\ref{pfactorization}) and (\ref{2factorization})}\\
&=&(\mathfrak{a}\mathfrak{p})\overline{(\mathfrak{a}\mathfrak{p})}\quad(=
(\mathfrak{a}\overline{\mathfrak{p}})\overline{(\mathfrak{a}\overline{\mathfrak{p}})}).
\end{eqnarray*}
So, without loss of generality we may assume that $\mathfrak{a}\mathfrak{p}=(2x_2+\sqrt{-14}y_2)\mathcal{O}_K$.
And, since $\mathfrak{p}$ is principal by (\ref{pfactorization}), we get that $\mathfrak{a}$
is also principal, say, $\mathfrak{a}=(u+\sqrt{-14}v)\mathcal{O}_K$ for some $u,v\in\mathbb{Z}$.
It then follows that
\begin{equation*}
2=N_{\mathcal{O}_K}(\mathfrak{a})=N_{K/\mathbb{Q}}(u+\sqrt{-14}v)=u^2+14v^2.
\end{equation*}
But there is no such pair of integers $u$ and $v$, which yields a contradiction. Hence we conclude
that
\begin{equation}\label{seperate}
\begin{array}{l}
\textrm{there is no odd prime $p$ not dividing $14$}\\
\textrm{which can be represented by both $x^2+14y^2$ and $2x^2+7y^2$}.
\end{array}
\end{equation}
\par
Therefore we achieve by Theorem \ref{main}, Proposition \ref{Legendre}, (\ref{twogenera}), (\ref{x2+14y2})
and (\ref{seperate}) that
for distinct odd primes $p$ and $q$ not dividing $14$
\begin{eqnarray*}
&&pq=x^2+14y^2~\textrm{for some}~x,y\in\mathbb{Z}\\
&\Longleftrightarrow&
\textrm{both $p$ and $q$ are represented by}~
\left\{\begin{array}{l}
x^2+14y^2,~\textrm{or}\\
2x^2+7y^2,~\textrm{or}\\
3x^2+2xy+5y^2
\end{array}\right.
\\
&\Longleftrightarrow&
\left\{\begin{array}{l}
p,q\equiv 1,9,15,23,25,39\pmod{56}~\textrm{and}~
p,q\in S,~\textrm{or}\\
p,q\equiv 1,9,15,23,25,39\pmod{56}~\textrm{and}~
p,q\not\in S,~\textrm{or}\\
p,q\equiv3,5,13,19,27,45\pmod{56}.
\end{array}\right.
\end{eqnarray*}
\end{example}

\bibliographystyle{amsplain}

\begin{thebibliography}{99}

\bibitem {Cox} D. A. Cox, \textit{Primes of the form $x^2+ny^2$: Fermat, Class Field, and Complex Multiplication},
John Wiley \& Sons, Inc., New York, 1989.

\bibitem {Gee} A. Gee, \textit{Class invariants by Shimura's reciprocity law},
J. Th\'{e}or. Nombres Bordeaux 11 (1999), no. 1, 45--72.

\bibitem {Janusz} G. J. Janusz, \textit{Algebraic Number Fields},
2nd edition, Grad. Studies in Math. 7, Amer. Math. Soc., Providence,
R. I., 1996.

\bibitem {Kani} E. Kani, \textit{Idoneal numbers and some generalizations},
Ann. Sci. Math. Qu\'{e}bec 35 (2011), no. 2, 197--227.

\bibitem {K-S-Y} J. K. Koo, D. H. Shin and D. S. Yoon, \textit{Ring class fields by smaller
generators}, http://arxiv.org/abs/1404.3282.

\bibitem {Lagrange} J. L. Lagrange, \textit{Oeuvres}, Vol. 3, Gauthier-Villars, Paris, 1869.

\bibitem {Lang} S. Lang, \textit{Elliptic Functions}, With an appendix by J. Tate, 2nd edition, Grad. Texts in Math. 112, Spinger-Verlag, New York, 1987.

\bibitem {Schertz} R. Schertz, \textit{Complex multiplication}, New Math. Monographs 15, Cambridge University Press, Cambridge, 2010.

\bibitem {Weil} A. Weil, \textit{Number Theory: An Approach Through History; From Hammurapi to Legendre}, Birkh\"{a}user Boston, Inc., Boston, MA, 1984.

\bibitem {Weinberger} P. J. Weinberger, \textit{Exponents of the class groups of complex
 quadratic fields}, Acta Arith. 22 (1973), 117--124.

\end{thebibliography}

\address{
Department of Mathematical Sciences \\
KAIST \\
Daejeon 305-701 \\
Republic of Korea} {jkkoo@math.kaist.ac.kr}
\address{
Department of Mathematics\\
Hankuk University of Foreign Studies\\
Yongin-si, Gyeonggi-do 449-791\\
Republic of Korea} {dhshin@hufs.ac.kr}
\end{document}